\documentclass[reqno,onecolumn,12pt]{amsart}
\usepackage{amsfonts,amscd,amsthm}
\usepackage{amssymb}
\usepackage{amsmath}
\usepackage{graphicx}
\usepackage{amscd,amsthm}

\newtheorem{theorem}{Theorem}
\theoremstyle{plain}

\newtheorem{corollary}{Corollary}

\newtheorem{definition}{Definition}

\newtheorem{remark}{Remark}

\numberwithin{equation}{section}

\input{tcilatex}

\begin{document}
\author{}
\title{}
\maketitle

\begin{center}
\pagestyle{myheadings}\thispagestyle{empty}%
\markboth{\bf \.{I}lkay Arslan Güven, Semra Kaya Nurkan and \.{I}pek A\u{g}ao\u{g}lu Tor}
{\bf NOTES ON W-DIRECTION CURVES IN EUCLIDEAN 3-SPACE}

\textbf{\large \ NOTES ON W-DIRECTION CURVES IN EUCLIDEAN 3-SPACE}

\bigskip

\textbf{\.{I}lkay Arslan G\"{u}ven$^{1,\ast }$, Semra Kaya Nurkan$^{2}$ and 
\.{I}pek A\u{g}ao\u{g}lu Tor}$^{\mathbf{3}}$

\bigskip

$^{\mathbf{1,3}}$Department of Mathematics, Faculty of Arts and Science, 
Gaziantep University, TR-27310 Gaziantep, Turkey

\textbf{E-Mail: }$^{\mathbf{1}}$\textbf{iarslan@gantep.edu.tr,
ilkayarslan81@hotmail.com}

\textbf{$^{\ast }$Corresponding Author}

$^{\mathbf{3}}$\textbf{agaogluipek@gmail.com}\\[2mm]

$^{\mathbf{2}}$Department of Mathematics, Faculty of Arts and Science, U\c{s}%
ak University,

TR-64200 U\c{s}ak, Turkey

\textbf{E-Mail: semra.kaya@usak.edu.tr, semrakaya\_gs@yahoo.com\\[2mm]
}

\textbf{\large Abstract}
\end{center}

\begin{quotation}
In this paper, we study the spherical indicatrices of W-direction curves in
three dimensional Euclidean space which were defined by using the unit
Darboux vector field $W$ of a Frenet curve, in \cite{Mac}. We obtain the
Frenet apparatus of these spherical indicatrix curves and the
characterizations of\ being general helix and slant helix. Moreover we give
some properties between the spherical indicatrix curves and their associated
curves.

Then we investigate two special ruled surface that are normal and binormal
surface by using W-direction curves. We give useful results involving the
characterizations of these ruled surfaces. From those applications, we make
use of such a work to interpret the Gaussian, mean curvatures of these
surfaces and geodesic, normal curvature and geodesic torsion of the base
curves with respect to these surfaces.
\end{quotation}



\section{\textbf{Introduction}}

The theory of curves is a subbranch of geometry which deals with curves in
Euclidean space or other spaces by using differential and integral calculus.
One of the most studied topic in curve theory is associated curves like
involute-evolute pairs, Bertrand curve pairs, Mannheim partner curves and
W-direction curves. Working with these associated curves is nice aspect that
these curves can be characterizated by the properties and behavior of the
main curves of them.

The most commonly used ones to characterize curves are general or
cylindrical helix and slant helix. A \textit{general helix }in\textit{\ }$%
\mathbb{E}^{3}$ is defined as; its tangent vector field makes a constant
angle with a fixed direction. If the principal normal vector field makes a
constant angle \ with a fixed direction, it is called \textit{slant helix}.
Izumiya and Takeuchi founded that a curve is a slant helix if and only if
the geodesic curvature of the principal image of the principal normal
indicatrix which is%
\begin{equation}
\delta (s)=\left( \frac{\kappa ^{2}}{(\kappa ^{2}+\tau ^{2})^{\frac{3}{2}}}%
.\left( \frac{\tau }{\kappa }\right) ^{\prime }\right) (s)  \tag{1.1}
\end{equation}%
is a constant function \cite{Izu2}.

Choi and Kim created the principal(binormal)-direction curves and
principal(binormal)-donor curve of a Frenet curve in $\mathbb{E}^{3}.$ They
gave the relation of curvature and torsion between the principal-direction
curve and its mate curve. They also defined a new curve \ called
PD-rectifying curve and gave a new characterization of a Bertrand curve by
means of the PD-rectifying curve. They made an application of associated
curves and studied a general helix and slant helix as principal-donor and
second principal-donor curve of a plane, respectively \cite{Choi}. Then Choi
et al. worked on the principal(binormal)-direction curve and
principal(binormal)-donor curve of a Frenet non-lightlike curves in $\mathbb{%
E}_{1}^{3}$ \ \cite{Choi2}.

After that K\"{o}rp\i nar et al. introduced associated curves according to
Bishop frame in $\mathbb{E}^{3}$ \ \cite{Kor}.

Recently, Macit and D\"{u}ld\"{u}l defined W-direction curve, W-rectifying
curve and V-direction curve of a Frenet curve in $\mathbb{E}^{3}$ and also
principal-direction curve, $B_{1}$-direction curve, $B_{2}$-direction curve
and $B_{2}$-rectifying curve in $\mathbb{E}^{4}$. These curves were given as
the integral curves of vector fields taken from the Frenet frame or Darboux
frame along a curve. They gave the relationship of the Frenet vector fields,
the curvature and the torsion between the associated curves and their main
curve \cite{Mac}.

The spherical indicatrix of a curve in three dimensional Euclidean space is
described by moving any of the unit Frenet vectors onto the unit sphere $\
S^{2}$. If any of the Frenet vectors is $X$ of a curve given with the
arc-lenght parameter $s$, then the equation of the spherical indicatrix
curve is given by%
\begin{equation}
\beta (s^{\ast })=X(s).  \tag{1.2}
\end{equation}%
where $s^{\ast }$ is the arc-length parameter of spherical indicatrix.

There are many works on spherical indicatrix curves. Kula and Yayl\i\ in 
\cite{Kula}, investigated spherical images of tangent and binormal
indicatrix of a slant helix. They found that the spherical images are
spherical helices.

Kula et al. gave some characterizations for a unit speed curve in $\mathbb{R}%
^{3}$ to being a slant helix by using its tangent, principal normal and
binormal indicatrix \cite{Kula2}.

Tun\c{c}er and \"{U}nal studied spherical indicatrices of a Bertrand curve
and its mate curve. They obtained relations between spherical images and new
representations of spherical indicatrices \cite{Tun}.

A \textit{ruled surface} in $\mathbb{R}^{3}$ is a surface which can be
described as the set of points swept out by moving a straight line in
surface. It therefore has a parametrization of the form%
\begin{equation}
\Phi (s,v)=\alpha (s)+v\delta (s)  \tag{1.3}
\end{equation}%
where $\alpha $ is a curve lying on the surface called \textit{base curve}
and $\delta $ is called \textit{director curve}. The straight lines are
called rulings. In using the equation of ruled surface we assume that $%
\alpha ^{\prime }$ is never zero and $\delta $ is not identically zero. The
rulings of ruled surface are asymptotic curves. Furthermore, the Gaussian
curvature of ruled surface is everywhere nonpositive. The ruled surface is
developable if and only if the distribution parameter vanishes and it is
minimal if and olny if its mean curvature vanishes \cite{Gray}. Recently,
the ruled surfaces were studied on Euclidean space and Minkowski space, see,
e.g., \cite{Ali, Izu, Izu2, Tur, Tur2, Yu}. Also Nurkan et al. investigated
two ruled surfaces such as normal and binormal surface by using any base
curve and adjoint curve of it \cite{Nur}.

In this paper, we study the spherical indicatrices of W-direction curves. We
obtain the Frenet apparatus of tangent indicatrix and binormal indicatrix
via Frenet vector fields, curvature and torsion of \ the main curves. We
give the characterizations of \ being general helix and slant helix of these
image curves. We also investigate the normal surface and the binormal
surface by taking the base curve as W-direction curve. We give usefull
results about developability, minimality of the surfaces and being
asymptotic line, geodesic curve, principal line of the base curves.

\section{Preliminaries}

Let $\beta :I\longrightarrow \mathbb{E}^{3}$ be a curve and $\left\{
T,N,B\right\} $ denote the Frenet frame of $\beta $. $T(s)=\beta ^{\prime
}(s)$ is called the \textit{unit tangent vector} of $\beta $ at $s$. $\beta $
is a unit speed curve (or parametrized by arc-length $s)$ if and only if $%
\left\Vert \beta ^{\prime }(s)\right\Vert =1.$ The \textit{curvature} of $%
\beta $ is given by $\kappa (s)=\left\Vert \beta ^{\prime \prime
}(s)\right\Vert $. The \textit{unit principal normal vector }$N(s)$ of $%
\beta $ at $s$ is given by $\beta ^{\prime \prime }(s)=\kappa (s).N(s)$.
Also the unit vector $B(s)=T(s)\times N(s)$ is called the \textit{unit
binormal vector }of $\beta $ at $s$. Then the famous Frenet formula holds as;%
\begin{eqnarray*}
T^{\prime }(s) &=&\kappa (s).N(s) \\
N^{\prime }(s) &=&-\kappa (s).T(s)+\tau (s)B(s) \\
B^{\prime }(s) &=&-\tau (s)N(s)
\end{eqnarray*}%
where $\tau (s)$ is the \textit{torsion} of $\beta $ at $s$ and calculated
as $\tau (s)=\langle N^{\prime }(s),B(s)\rangle $ $\ $or \ $\tau
(s)=\left\Vert B^{\prime }(s)\right\Vert .$

Also the Frenet vectors of a curve $\beta $, which is given by arc-length
parameter $s$, can be calculated as;%
\begin{eqnarray}
T(s) &=&\beta ^{\prime }(s)  \notag \\
N(s) &=&\frac{\beta ^{\prime \prime }(s)}{\left\Vert \beta ^{\prime \prime
}(s)\right\Vert }  \TCItag{2.1} \\
B(s) &=&T(s)\times N(s).  \notag
\end{eqnarray}%
For the unit speed curve $\beta :I\longrightarrow \mathbb{E}^{3}$ , the
vector 
\begin{equation}
W(s)=\tau (s)T(s)+\kappa (s)B(s)  \tag{2.2}
\end{equation}%
is called the Darboux vector of $\beta $ which is the rotation vector of
trihedron of the curve with curvature $\kappa \neq 0$ \ when a point moves
along the curve $\beta .$

A unit speed curve $\beta :I\longrightarrow \mathbb{E}^{n}$ \ is a \textit{%
Frenet curve} if $\beta ^{\prime \prime }(s)\neq 0$ and it has the non-zero
curvature.

\begin{definition}
Let $\beta $ be a Frenet curve in $\mathbb{E}^{3}$ \ and \ $W$ be the unit
Darboux vector field of $\beta .$ The integral curve of $W(s)$ is called
W-direction curve of $\beta .$ Namely, if $\overline{\beta }$ is the
W-direction curve of $\beta $, then $W(s)=\overline{\beta }^{\prime }(s)$,
where $W=\frac{1}{\sqrt{\kappa ^{2}+\tau ^{2}}}(\tau T+\kappa B)$ \ \ \cite%
{Mac}.
\end{definition}

Let the Frenet apparatus of a Frenet curve $\beta $ and its W-direction
curve be $\left\{ T,N,B,\kappa ,\tau \right\} $ \ and \ $\left\{ \overline{T}%
,\overline{N},\overline{B},\overline{\kappa },\overline{\tau }\right\} $
respectively. The relations of Frenet apparatus between the main curve and
W-direction curve are given in \cite{Mac} \ as;%
\begin{eqnarray}
\overline{T} &=&\frac{\tau }{\sqrt{\kappa ^{2}+\tau ^{2}}}T+\frac{\kappa }{%
\sqrt{\kappa ^{2}+\tau ^{2}}}B  \notag \\
\overline{N} &=&-\frac{\kappa }{\sqrt{\kappa ^{2}+\tau ^{2}}}T+\frac{\tau }{%
\sqrt{\kappa ^{2}+\tau ^{2}}}B  \TCItag{2.3} \\
\overline{B} &=&-N  \notag \\
\overline{\kappa } &=&\frac{\left\vert \tau \kappa ^{\prime }-\tau ^{\prime
}\kappa \right\vert }{\kappa ^{2}+\tau ^{2}}\text{ \ \ \ , \ \ \ \ }%
\overline{\tau }=\sqrt{\kappa ^{2}+\tau ^{2}}  \notag
\end{eqnarray}

\begin{remark}
In this paper we take the signs of absolute value positive. If the sign will
be taken negative, the expressions similarly have the other signs.
\end{remark}

\begin{theorem}
Let $\overline{\beta }$ be the W-direction curve of $\beta $ which is not a
general helix. Then $\overline{\beta }$ is a general helix if and only if $%
\beta $ is a slant helix \cite{Mac}.
\end{theorem}

\begin{theorem}
A curve is a \textit{general helix }if and only if $\frac{\tau }{\kappa }%
=cons$tant \cite{On}.
\end{theorem}

A ruled surface given with the equation (1.3) is in the general form. The 
\textit{normal} and the \textit{binormal} surfaces which are also ruled
surfaces are defined by%
\begin{eqnarray}
\Phi (s,v) &=&\alpha (s)+vN(s)  \TCItag{2.4} \\
\Phi (s,v) &=&\alpha (s)+vB(s)  \TCItag{2.5}
\end{eqnarray}%
where $\alpha $ is a curve with arc-length parameter $s$ and $\left\{
T,N,B\right\} $ are the Frenet vector fields of $\alpha $ \cite{Izu, Izu2}.

The \textit{distribution parameter} $\lambda $ of a ruled surface $\Phi $
given in equation (1.3) is dedicated as;%
\begin{equation}
\lambda =\frac{\det (\frac{d\alpha }{ds},\delta ,\frac{d\delta }{ds})}{%
\left\Vert \frac{d\delta }{ds}\right\Vert ^{2}}.  \tag{2.6}
\end{equation}

The standard unit normal vector field $U$ on the ruled surface $\Phi $ \ is
defined by%
\begin{equation}
U=\frac{\Phi _{s}\times \Phi _{v}}{\left\Vert \Phi _{s}\times \Phi
_{v}\right\Vert }.  \tag{2.7}
\end{equation}%
where $\Phi _{s}=\frac{d\Phi }{ds}$ \ and \ $\Phi _{v}=\frac{d\Phi }{dv}$.

The \textit{Gaussian curvature} and \textit{mean curvature} of a ruled
surface given in equation (1.3) \ are given respectively by%
\begin{equation}
K=\frac{eg-f^{2}}{EG-F^{2}}  \tag{2.8}
\end{equation}%
and%
\begin{equation}
H=\frac{Eg+Ge-2Ff}{2(EG-F^{2})}  \tag{2.9}
\end{equation}%
where $E=\left\langle \Phi _{s},\Phi _{s}\right\rangle $ \ , \ $%
F=\left\langle \Phi _{s},\Phi _{v}\right\rangle $ \ , \ $G=\left\langle \Phi
_{v},\Phi _{v}\right\rangle $ \ , \ $e=\left\langle \Phi
_{ss},U\right\rangle $ \ , \ $f=\left\langle \Phi _{sv},U\right\rangle $ \
and $g=\left\langle \Phi _{vv},U\right\rangle $ (see\cite{On})$.$

\begin{definition}
A surface is developable (flat) provided its Gaussian curvature is zero and
minimal provided its mean curvature is zero \cite{On}.
\end{definition}

The equation between the distribution parameter and Gaussian curvature which
is called \ Lamarle formula is defined as%
\begin{equation*}
K=-\frac{\lambda ^{2}}{\left( \lambda ^{2}+v^{2}\right) ^{2}}
\end{equation*}%
where $K$ is the Gaussian curvature and $\lambda $ is the distribution
parameter (see \cite{Krup}).

Namely, if the distribution parameter is zero, then the surface is
developable.

Let $\alpha $ be any curve of arc-length parameter $s$ and has the Frenet
frame $\left\{ T,N,B\right\} $ along $\alpha .$ If $\alpha $ is on a
surface, the frame $\left\{ T,V,U\right\} $ along the curve $\alpha $ is
called the Darboux frame where $T$ is the unit tangent vector of $\alpha $, $%
U$ is the unit normal of the surface and $V$ is the unit vector given by $%
V=U\times T.$ The relations between these vectors and their derivatives are 
\cite{On};%
\begin{equation*}
\left[ 
\begin{array}{c}
T^{\prime } \\ 
V^{\prime } \\ 
U^{\prime }%
\end{array}%
\right] =\left[ 
\begin{array}{ccc}
0 & \kappa _{g} & \kappa _{n} \\ 
-\kappa _{g} & 0 & \tau _{g} \\ 
-\kappa _{n} & -\tau _{g} & 0%
\end{array}%
\right] \left[ 
\begin{array}{c}
T \\ 
V \\ 
U%
\end{array}%
\right]
\end{equation*}%
where $\kappa _{g}$ is the \textit{geodesic curvature,} $\kappa _{n}$ is the 
\textit{normal curvature} and $\tau _{g}$ is the \textit{geodesic torsion.}

If the curve $\alpha $ is the base curve of the ruled surface given in
(1.3), then the geodesic curvature, normal curvature and geodesic torsion
with respect to the ruled surface are also given respectively by \cite{On};%
\begin{equation}
\kappa _{g}=\left\langle U\times T,T^{\prime }\right\rangle   \tag{2.10}
\end{equation}%
\begin{equation}
\kappa _{n}=\left\langle \alpha ^{\prime \prime },U\right\rangle   \tag{2.11}
\end{equation}%
and%
\begin{equation}
\tau _{g}=\left\langle U\times U^{\prime },T^{\prime }\right\rangle . 
\tag{2.12}
\end{equation}

For a curve $\alpha $ which is lying on a surface, the following statements
are satisfied \cite{On};

\textbf{1) }$\alpha $ is a geodesic curve if and only if the geodesic
curvature of the curve with respect to the surface vanishes.

\textbf{2) }$\alpha $ is a asymptotic line if and only if the normal
curvature of the curve with respect to the surface vanishes.

\textbf{3) }$\alpha $ is a principal line if and only if the geodesic
torsion of the curve with respect to the surface vanishes.

\section{\textbf{Spherical Images of W-direction Curves}}

In this section, we will introduce tangent indicatrix and binormal
indicatrix curves which are spherical indicatrices of W-direction curves. We
find their Frenet apparatus and give some results of being general helix and
slant helix.

Let $\beta $ be a curve with arc-length parameter $s$ and $\overline{\beta }$
be the W-direction curve of $\beta .$ The arc-length parameter $\overline{s}$
of $\overline{\beta }$ which is an integral curve of $\beta $, can be taken
as $\overline{s}=s$ (see \cite{Choi}). The Frenet apparatus of $\beta $ and $%
\overline{\beta }$ are $\left\{ T,N,B,\kappa ,\tau \right\} $ \ and \ $%
\left\{ \overline{T},\overline{N},\overline{B},\overline{\kappa },\overline{%
\tau }\right\} $ respectively. Here also $\delta $ is the geodesic curvature
of the principal image of the principal normal indicatrix given with the
equation (1.1).

Using the equation (1.2), the tangent indicatrix and binormal indicatrix
curves of W-direction curve $\overline{\beta }$ are given with the equations%
\begin{eqnarray}
\alpha (s_{\alpha }) &=&\overline{T}(s)  \TCItag{3.1} \\
\gamma (s_{\gamma }) &=&\overline{B}(s)  \TCItag{3.2}
\end{eqnarray}%
where $s_{\alpha }$ \ and \ $s_{\gamma }$ \ are arc-length parameters of
tangent and binormal indicatrix curves, respectively.

\begin{theorem}
Let $\beta $ be a curve with arc-length parameter $s$ and $\overline{\beta }$
be the W-direction curve of $\beta .$ The Frenet vector fields, curvature
and torsion of the tangent indicatrix curve $\alpha $ of W-direction curve\
are given by%
\begin{eqnarray*}
T_{\alpha } &=&-\frac{1}{\sqrt{1+f^{2}}}T+\frac{f}{\sqrt{1+f^{2}}}B \\
N_{\alpha } &=&-\frac{\kappa ^{\prime }(f-g)}{\sqrt{(\kappa ^{\prime
})^{2}(f-g)^{2}(1+f^{2})+\kappa ^{4}(1+f^{2})^{4}}}\left( fT+\frac{\kappa
^{2}(1+f^{2})^{2}}{\kappa ^{\prime }(f-g)}N+B\right)  \\
B_{\alpha } &=&\frac{\kappa (1+f^{2})}{\sqrt{(\kappa ^{\prime
})^{2}(f-g)^{2}+\kappa ^{4}(1+f^{2})^{3}}}\left( \tau T-\frac{\kappa
^{\prime }(f-g)}{\kappa (1+f^{2})}N+\kappa B\right)  \\
\kappa _{\alpha } &=&\sqrt{1+\frac{\kappa ^{4}(1+f^{2})^{3}}{(\kappa
^{\prime })^{2}(f-g)^{2}}} \\
\tau _{\alpha } &=&\frac{\kappa ^{2}\sqrt{(1+f^{2})^{3}}}{\kappa ^{\prime
}(f-g)\left( (\kappa ^{\prime })^{2}(f-g)^{2}+\kappa
^{4}(1+f^{2})^{3}\right) }\left( 3(\kappa ^{\prime })^{2}(1+fg)(f-g)-\kappa
\kappa ^{\prime \prime }(f-h)(1+f^{2})\right) 
\end{eqnarray*}%
where $f=\frac{\tau }{\kappa }$, \ \ $g=\frac{\tau ^{\prime }}{\kappa
^{\prime }}$ \ and \ $h=\frac{\tau ^{\prime \prime }}{\kappa ^{\prime \prime
}}.$
\end{theorem}

\begin{proof}
The equation of tangent indicatrix curve $\alpha $ \ is given in equation
(3.1) with the arc-length parameter $s_{\alpha }$. By differentiating
equation (3.1) and using Frenet formulas we get%
\begin{equation*}
\frac{ds_{\alpha }}{ds}=\overline{\kappa }.
\end{equation*}%
If we use the equations in (2.1) and the relation $\frac{ds_{\alpha }}{ds}=%
\overline{\kappa }$ , we find the tangent, prinicipal and binormal vector
fields respectively as;%
\begin{eqnarray*}
T_{\alpha } &=&\overline{N} \\
N_{\alpha } &=&-\frac{\overline{\kappa }}{\sqrt{\overline{\kappa }^{2}+%
\overline{\tau }^{2}}}\overline{T}+\frac{\overline{\tau }}{\sqrt{\overline{%
\kappa }^{2}+\overline{\tau }^{2}}}\overline{B} \\
B_{\alpha } &=&\frac{\overline{\tau }}{\sqrt{\overline{\kappa }^{2}+%
\overline{\tau }^{2}}}\overline{T}+\frac{\overline{\kappa }}{\sqrt{\overline{%
\kappa }^{2}+\overline{\tau }^{2}}}\overline{B}.
\end{eqnarray*}%
By writing the relations (2.3) in the last equations, we have;%
\begin{eqnarray*}
T_{\alpha } &=&-\frac{\kappa }{\sqrt{\kappa ^{2}+\tau ^{2}}}T+\frac{\tau }{%
\sqrt{\kappa ^{2}+\tau ^{2}}}B \\
N_{\alpha } &=&-\frac{\tau \kappa ^{\prime }-\tau ^{\prime }\kappa }{\sqrt{%
(\tau \kappa ^{\prime }-\tau ^{\prime }\kappa )^{2}+(\kappa ^{2}+\tau
^{2})^{3}}}\left( \frac{\tau }{\sqrt{\kappa ^{2}+\tau ^{2}}}T+\frac{\sqrt{%
(\kappa ^{2}+\tau ^{2})^{3}}}{\tau \kappa ^{\prime }-\tau ^{\prime }\kappa }%
N+\frac{\kappa }{\sqrt{\kappa ^{2}+\tau ^{2}}}B\right) \\
B_{\alpha } &=&\frac{\kappa ^{2}+\tau ^{2}}{\sqrt{(\tau \kappa ^{\prime
}-\tau ^{\prime }\kappa )^{2}+(\kappa ^{2}+\tau ^{2})^{3}}}\left( \tau T-%
\frac{\tau \kappa ^{\prime }-\tau ^{\prime }\kappa }{\kappa ^{2}+\tau ^{2}}%
N+\kappa B\right) .
\end{eqnarray*}%
In the last equations assuming $f=\frac{\tau }{\kappa }$, \ \ $g=\frac{\tau
^{\prime }}{\kappa ^{\prime }}$ \ and \ $h=\frac{\tau ^{\prime \prime }}{%
\kappa ^{\prime \prime }}$ \ and \ arranging the expressions, we obtain the
tangent, principal and binormal vector fields of the tangent indicatrix
curve $\alpha $ of the W-direction curve $\overline{\beta },$ with respect
to the main curve $\beta .$

Also the curvature and torsion of the tangent indicatrix curve $\alpha $ are
found as;%
\begin{eqnarray}
\kappa _{\alpha } &=&\frac{\sqrt{\overline{\kappa }^{2}+\overline{\tau }^{2}}%
}{\overline{\kappa }}  \TCItag{3.3} \\
\tau _{\alpha } &=&\sqrt{\frac{(\overline{\tau }^{\prime }\overline{\kappa }-%
\overline{\tau }\overline{\kappa }^{\prime })^{2}+\left( \overline{\kappa }%
^{\prime }\frac{\overline{\tau }^{2}}{\overline{\kappa }}-\overline{\tau }%
\overline{\tau }^{\prime }\right) ^{2}}{\left( \overline{\kappa }^{2}+%
\overline{\tau }^{2}\right) ^{3}}}.  \TCItag{3.4}
\end{eqnarray}%
Again by using the relations (2.3) in equations (3.3) and (3.4), we get;%
\begin{eqnarray*}
\kappa _{\alpha } &=&\frac{\sqrt{(\tau \kappa ^{\prime }-\tau ^{\prime
}\kappa )^{2}+(\kappa ^{2}+\tau ^{2})^{3}}}{\tau \kappa ^{\prime }-\tau
^{\prime }\kappa } \\
\tau _{\alpha } &=&\frac{\sqrt{(\kappa ^{2}+\tau ^{2})^{3}}}{(\tau \kappa
^{\prime }-\tau ^{\prime }\kappa )\left( (\tau \kappa ^{\prime }-\tau
^{\prime }\kappa )^{2}+(\kappa ^{2}+\tau ^{2})^{3}\right) }\left( 
\begin{array}{c}
3(\kappa \kappa ^{\prime }+\tau \tau ^{\prime })(\tau \kappa ^{\prime }-\tau
^{\prime }\kappa ) \\ 
-(\tau \kappa ^{\prime \prime }-\tau ^{\prime \prime }\kappa )(\kappa
^{2}+\tau ^{2})%
\end{array}%
\right) .
\end{eqnarray*}%
Taking $f,$ $g$ \ and \ $h$ \ in the last equations, we reach the result.
\end{proof}

\begin{theorem}
If any curve $\beta $ with arc-length parameter $s$ is slant helix, then the
tangent indicatrix of W-direction curve of $\beta $ is a planar curve.
\end{theorem}

\begin{proof}
By Theorem 6 in \cite{Mac}, if the geodesic curvature of the principal image
of the principal normal indicatrix of the curve $\beta $ is $\delta $, then $%
\frac{\overline{\kappa }}{\overline{\tau }}=\delta $ where $\overline{\kappa 
}$ \ and $\overline{\tau }$ are curvature and torsion of the W-direction
curve of $\beta .$

By using the equation (3.4) and $\frac{\overline{\kappa }}{\overline{\tau }}%
=\delta $, the torsion of the tangent indicatrix curve $\alpha $ is found as;%
\begin{equation*}
\tau _{\alpha }=\frac{\delta ^{\prime }}{\overline{\tau }\text{ }\delta
(1+\delta ^{2})}.
\end{equation*}%
If the curve $\beta $ is slant helix, then $\delta ^{\prime }=0$. So by the
last equation $\tau _{\alpha }=0$ which means the tangent indicatrix curve $%
\alpha $ is planar.
\end{proof}

\begin{remark}
From now on, the equations $f=\frac{\tau }{\kappa }$, \ \ $g=\frac{\tau
^{\prime }}{\kappa ^{\prime }}$ , \ $h=\frac{\tau ^{\prime \prime }}{\kappa
^{\prime \prime }}$ \ , \ $\delta =\frac{\overline{\kappa }}{\overline{\tau }%
}$ \ and $\kappa ^{\prime }(f-g)=-\delta \kappa ^{2}(1+f^{2})^{3/2}$ \ will
be used in the calculations.
\end{remark}

\begin{theorem}
The tangent indicatrix curve $\alpha $ of the W-direction curve is a general
helix if and only if the following equation is satisfied;%
\begin{equation*}
A^{\prime }-A\left( \frac{3\kappa ^{\prime }}{\kappa }+\frac{4ff^{\prime }}{%
1+f^{2}}+\frac{3\delta \delta ^{\prime }}{1+\delta ^{2}}\right) =0
\end{equation*}%
where \ $A=\kappa ^{\prime \prime }(f-h)+3\kappa \kappa ^{\prime }\sqrt{%
1+f^{2}}(1+fg)\delta .$
\end{theorem}

\begin{proof}
If we take ratio of the torsion and curvature of the tangent indicatrix
which are in Theorem 3 , use the relation $\kappa ^{\prime }(f-g)=-\delta
\kappa ^{2}(1+f^{2})^{3/2}$ \ and make some appropriate calculations, we
have;%
\begin{equation}
\frac{\tau _{\alpha }}{\kappa _{\alpha }}=-\frac{A}{\kappa
^{3}(1+f^{2})^{2}(1+\delta ^{2})^{3/2}}.  \tag{3.5}
\end{equation}%
By differentiating the equation (3.5), we find that;%
\begin{equation*}
\left( \frac{\tau _{\alpha }}{\kappa _{\alpha }}\right) ^{\prime }=-\frac{%
A^{\prime }-A\left( \frac{3\kappa ^{\prime }}{\kappa }+\frac{4ff^{\prime }}{%
1+f^{2}}+\frac{3\delta \delta ^{\prime }}{1+\delta ^{2}}\right) }{\kappa
^{3}(1+f^{2})^{2}(1+\delta ^{2})^{3/2}}.
\end{equation*}%
If the numerator of the last fraction is zero, then $\left( \frac{\tau
_{\alpha }}{\kappa _{\alpha }}\right) ^{\prime }=0.$ Since the harmonic
curvature of the curve $\alpha $ is constant and it is a general helix.
\end{proof}

\begin{corollary}
If the curve $\beta $ is a general helix and and the equation $A^{\prime }-3A%
\frac{\kappa ^{\prime }}{\kappa }=0$ \ is satisfied, then the tangent
indicatrix of the W-direction curve $\overline{\beta }$ is a general helix.
\end{corollary}

\begin{proof}
If $\beta $ is a general helix, then $f$ is constant and also $f^{\prime
}=0. $ \ Since $f^{\prime }=0$, then \ $\delta =0.$ For the derivative, we
find;%
\begin{equation*}
\left( \frac{\tau _{\alpha }}{\kappa _{\alpha }}\right) ^{\prime }=-\frac{%
A^{\prime }-A\frac{3\kappa ^{\prime }}{\kappa }}{\kappa
^{3}(1+f^{2})^{2}(1+\delta ^{2})^{3/2}}.
\end{equation*}%
If the numerator of this fraction is zero, then we reach the result clearly.
\end{proof}

\begin{theorem}
The tangent indicatrix curve $\alpha $ of the W-direction curve is a slant
helix if and only if the following equation is satisfied;%
\begin{equation*}
\delta ^{\prime }(1+4\delta ^{2})\frac{X}{Y}+\delta (1+\delta ^{2})\left( 
\frac{X}{Y}\right) ^{\prime }=0
\end{equation*}%
where $\delta =$\ $\frac{\overline{\kappa }}{\overline{\tau }}$ $\ ,$ $\
X=(1+\delta ^{2})\overline{\tau }(\delta ^{\prime \prime }\overline{\tau }%
-\delta ^{\prime }\overline{\tau }^{\prime })-3\delta (\delta ^{\prime })^{2}%
\overline{\tau }^{2}$ $\ $and $\ Y=\left( \overline{\tau }^{2}(1+\delta
^{2})^{3}+(\delta ^{\prime })^{2}\right) ^{3/2}.$
\end{theorem}

\begin{proof}
From the equation (1.1), the geodesic curvature of the principal image of
the principal normal indicatrix of the tangent indicatrix curve $\alpha $ is
given by%
\begin{equation}
\delta _{\alpha }=\frac{\kappa _{\alpha }^{2}}{(\kappa _{\alpha }^{2}+\tau
_{\alpha }^{2})^{\frac{3}{2}}}.\left( \frac{\tau _{\alpha }}{\kappa _{\alpha
}}\right) ^{\prime }.  \tag{3.6}
\end{equation}%
We take into account that \ $\overline{\tau }^{\prime }\overline{\kappa }-%
\overline{\tau }\overline{\kappa }^{\prime }=-\overline{\tau }^{2}\delta
^{\prime }$ \ , \ \ $\overline{\kappa }^{\prime }\frac{\overline{\tau }^{2}}{%
\overline{\kappa }}-\overline{\tau }\overline{\tau }^{\prime }=\overline{%
\tau }^{2}\frac{\delta ^{\prime }}{\delta }$ \ and \ $\overline{\kappa }^{2}+%
\overline{\tau }^{2}=\overline{\tau }^{2}(1+\delta ^{2})$\ , then put the
equations (3.3) and (3.4) in (3.6), we clearly find%
\begin{equation*}
\delta _{\alpha }=\frac{1+\delta ^{2}}{\delta ^{2}}\left( \frac{\delta
^{2}(1+\delta ^{2})^{2}\overline{\tau }^{2}}{(1+\delta ^{2})^{3}\overline{%
\tau }^{2}+(\delta ^{\prime })^{2}}\right) ^{3/2}\left( \frac{\delta
^{\prime }}{\overline{\tau }(1+\delta ^{2})^{3/2}}\right) ^{\prime }.
\end{equation*}%
After some calculations we have%
\begin{equation}
\delta _{\alpha }=\frac{\delta (1+\delta ^{2})^{3/2}.X}{Y}  \tag{3.7}
\end{equation}%
where $X=(1+\delta ^{2})\overline{\tau }(\delta ^{\prime \prime }\overline{%
\tau }-\delta ^{\prime }\overline{\tau }^{\prime })-3\delta (\delta ^{\prime
})^{2}\overline{\tau }^{2}$ \ \ and $\ Y=\left( \overline{\tau }%
^{2}(1+\delta ^{2})^{3}+(\delta ^{\prime })^{2}\right) ^{3/2}.$

By differentiating the equation (3.7), we obtain%
\begin{equation*}
\left( \delta _{\alpha }\right) ^{\prime }=(1+\delta ^{2})^{1/2}\left(
\delta ^{\prime }(1+4\delta ^{2})\frac{X}{Y}+\delta (1+\delta ^{2})\left( 
\frac{X}{Y}\right) ^{\prime }\right) .
\end{equation*}%
$\left( \delta _{\alpha }\right) ^{\prime }=0$ \ if and only if $\delta
^{\prime }(1+4\delta ^{2})\frac{X}{Y}+\delta (1+\delta ^{2})\left( \frac{X}{Y%
}\right) ^{\prime }=0.$

So by taking into consideration the equation (1.1) , the proof is completed.
\end{proof}

\begin{theorem}
Let $\beta $ be a curve with arc-length parameter $s$ and $\overline{\beta }$
be the W-direction curve of $\beta .$ The Frenet vector fields, curvature
and torsion of the binormal indicatrix curve $\gamma $ of W-direction curve\
are given by%
\begin{eqnarray*}
T_{\gamma } &=&\frac{1}{\sqrt{1+f^{2}}}T-\frac{f}{\sqrt{1+f^{2}}}B \\
N_{\gamma } &=&\frac{\kappa ^{\prime }(f-g)}{\sqrt{(\kappa ^{\prime
})^{2}(f-g)^{2}(1+f^{2})+\kappa ^{4}(1+f^{2})^{4}}}\left( fT+\frac{\kappa
^{2}(1+f^{2})^{2}}{\kappa ^{\prime }(f-g)}N+B\right)  \\
B_{\gamma } &=&\frac{\kappa (1+f^{2})}{\sqrt{(\kappa ^{\prime
})^{2}(f-g)^{2}+\kappa ^{4}(1+f^{2})^{3}}}\left( \tau T-\frac{\kappa
^{\prime }(f-g)}{\kappa (1+f^{2})}N+\kappa B\right)  \\
\kappa _{\gamma } &=&\sqrt{1+\frac{(\kappa ^{\prime })^{2}(f-g)^{2}}{\kappa
^{4}(1+f^{2})^{3}}} \\
\tau _{\gamma } &=&\frac{1}{(\kappa ^{\prime })^{2}(f-g)^{2}+\kappa
^{4}(1+f^{2})^{3}}\left( 3(\kappa ^{\prime })^{2}(1+fg)(f-g)-\kappa \kappa
^{\prime \prime }(f-h)(1+f^{2})\right) 
\end{eqnarray*}%
where $f=\frac{\tau }{\kappa }$, \ \ $g=\frac{\tau ^{\prime }}{\kappa
^{\prime }}$ \ and \ $h=\frac{\tau ^{\prime \prime }}{\kappa ^{\prime \prime
}}.$
\end{theorem}

\begin{proof}
The equation of binormal indicatrix curve $\gamma $ \ is given in equation
(3.2) with the arc-length parameter $s_{\gamma }$. By differentiating
equation (3.2) and using Frenet formulas we get%
\begin{equation*}
\frac{ds_{\gamma }}{ds}=\left\vert \overline{\tau }\right\vert .
\end{equation*}%
Here we assume $\overline{\tau }\rangle 0$. \ If \ $\overline{\tau }\langle
0,$ the Frenet vector fields have other signs.

If we use the equations in (2.1) and the relation $\frac{ds_{\gamma }}{ds}=%
\overline{\tau }$, we find the tangent, prinicipal and binormal vector
fields respectively as;%
\begin{eqnarray*}
T_{\gamma } &=&-\overline{N} \\
N_{\gamma } &=&\frac{\overline{\kappa }}{\sqrt{\overline{\kappa }^{2}+%
\overline{\tau }^{2}}}\overline{T}-\frac{\overline{\tau }}{\sqrt{\overline{%
\kappa }^{2}+\overline{\tau }^{2}}}\overline{B} \\
B_{\gamma } &=&\frac{\overline{\tau }}{\sqrt{\overline{\kappa }^{2}+%
\overline{\tau }^{2}}}\overline{T}+\frac{\overline{\kappa }}{\sqrt{\overline{%
\kappa }^{2}+\overline{\tau }^{2}}}\overline{B}.
\end{eqnarray*}%
By writing the relations (2.3) in the last equations, we have;%
\begin{eqnarray*}
T_{\gamma } &=&\frac{\kappa }{\sqrt{\kappa ^{2}+\tau ^{2}}}T-\frac{\tau }{%
\sqrt{\kappa ^{2}+\tau ^{2}}}B \\
N_{\gamma } &=&\frac{\tau \kappa ^{\prime }-\tau ^{\prime }\kappa }{\sqrt{%
(\tau \kappa ^{\prime }-\tau ^{\prime }\kappa )^{2}+(\kappa ^{2}+\tau
^{2})^{3}}}\left( \frac{\tau }{\sqrt{\kappa ^{2}+\tau ^{2}}}T+\frac{\sqrt{%
(\kappa ^{2}+\tau ^{2})^{3}}}{\tau \kappa ^{\prime }-\tau ^{\prime }\kappa }%
N+\frac{\kappa }{\sqrt{\kappa ^{2}+\tau ^{2}}}B\right) \\
B_{\gamma } &=&\frac{\kappa ^{2}+\tau ^{2}}{\sqrt{(\tau \kappa ^{\prime
}-\tau ^{\prime }\kappa )^{2}+(\kappa ^{2}+\tau ^{2})^{3}}}\left( \tau T-%
\frac{\tau \kappa ^{\prime }-\tau ^{\prime }\kappa }{\kappa ^{2}+\tau ^{2}}%
N+\kappa B\right) .
\end{eqnarray*}%
In the last equations assuming $f=\frac{\tau }{\kappa }$, \ \ $g=\frac{\tau
^{\prime }}{\kappa ^{\prime }}$ \ and \ $h=\frac{\tau ^{\prime \prime }}{%
\kappa ^{\prime \prime }}$ \ , \ arranging the expressions , we obtain the
tangent, principal and binormal vector fields of the binormal indicatrix
curve $\gamma $ of the W-direction curve $\overline{\beta },$ with respect
to the main curve $\beta .$

Also the curvature and torsion of the binormal indicatrix curve $\gamma $ by
taking into account that $\frac{ds_{\gamma }}{ds}=\overline{\tau }$ are
found as;%
\begin{eqnarray}
\kappa _{\gamma } &=&\frac{\sqrt{\overline{\kappa }^{2}+\overline{\tau }^{2}}%
}{\overline{\tau }}  \TCItag{3.8} \\
\tau _{\gamma } &=&\sqrt{\frac{(\overline{\tau }\overline{\kappa }^{\prime }-%
\overline{\tau }^{\prime }\overline{\kappa })^{2}+\left( \overline{\tau }%
^{\prime }\frac{\overline{\kappa }^{2}}{\overline{\tau }}-\overline{\kappa }%
\overline{\kappa }^{\prime }\right) ^{2}}{\left( \overline{\kappa }^{2}+%
\overline{\tau }^{2}\right) ^{3}}}.  \TCItag{3.9}
\end{eqnarray}%
Again by using the relations (2.3) in equations (3.8) and (3.9), we get;%
\begin{eqnarray*}
\kappa _{\gamma } &=&\frac{\sqrt{(\tau \kappa ^{\prime }-\tau ^{\prime
}\kappa )^{2}+(\kappa ^{2}+\tau ^{2})^{3}}}{\sqrt{(\kappa ^{2}+\tau ^{2})^{3}%
}} \\
\tau _{\gamma } &=&\frac{1}{(\tau \kappa ^{\prime }-\tau ^{\prime }\kappa
)^{2}+(\kappa ^{2}+\tau ^{2})^{3}}\left( 
\begin{array}{c}
3(\kappa \kappa ^{\prime }+\tau \tau ^{\prime })(\tau \kappa ^{\prime }-\tau
^{\prime }\kappa ) \\ 
-(\tau \kappa ^{\prime \prime }-\tau ^{\prime \prime }\kappa )(\kappa
^{2}+\tau ^{2})%
\end{array}%
\right) .
\end{eqnarray*}%
Taking $f,$ $g$ \ and \ $h$ \ in the last equations, we reach the result.
\end{proof}

\begin{corollary}
The tangent indicatrix curve and the binormal indicatrix curve of a
W-direction curve are Bertrand mate curves.
\end{corollary}

\begin{proof}
Since the relation between the principal normal vector fields of the tangent
indicatrix and binormal indicatrix is;%
\begin{equation*}
N_{\alpha }=-N_{\gamma },
\end{equation*}%
they are linearly dependent.
\end{proof}

\begin{theorem}
If any curve $\beta $ with arc-length parameter $s$ is slant helix, then the
binormal indicatrix of W-direction curve of $\beta $ is a planar curve.
\end{theorem}

\begin{proof}
By using the equation (3.9) and $\frac{\overline{\kappa }}{\overline{\tau }}%
=\delta $, the torsion of the binormal indicatrix curve $\gamma $ is found
as;%
\begin{equation*}
\tau _{\gamma }=\frac{\delta ^{\prime }}{\overline{\tau }\text{ }(1+\delta
^{2})}.
\end{equation*}%
If the curve $\beta $ is slant helix, then $\delta ^{\prime }=0$. So by the
last equation $\tau _{\gamma }=0$ which means the binormal indicatrix curve $%
\gamma $ is planar.
\end{proof}

\begin{theorem}
The binormal indicatrix curve $\gamma $ of the W-direction curve is a
general helix if and only if the following equation is satisfied;%
\begin{equation*}
A^{\prime }-A\left( \frac{3\kappa ^{\prime }}{\kappa }+\frac{4ff^{\prime }}{%
1+f^{2}}+\frac{3\delta \delta ^{\prime }}{1+\delta ^{2}}\right) =0
\end{equation*}%
where \ $A=\kappa ^{\prime \prime }(f-h)+3\kappa \kappa ^{\prime }\sqrt{%
1+f^{2}}(1+fg)\delta .$
\end{theorem}

\begin{proof}
If we take ratio of the torsion and curvature of the binormal indicatrix
which are in Theorem 7 , use the relation $\kappa ^{\prime }(f-g)=-\delta
\kappa ^{2}(1+f^{2})^{3/2}$ \ and make some appropriate calculations, we
have;%
\begin{equation}
\frac{\tau _{\gamma }}{\kappa _{\gamma }}=-\frac{A}{\kappa
^{3}(1+f^{2})^{2}(1+\delta ^{2})^{3/2}}.  \tag{3.10}
\end{equation}%
By differentiating the equation (3.10), we find that;%
\begin{equation*}
\left( \frac{\tau _{\gamma }}{\kappa _{\gamma }}\right) ^{\prime }=-\frac{%
A^{\prime }-A\left( \frac{3\kappa ^{\prime }}{\kappa }+\frac{4ff^{\prime }}{%
1+f^{2}}+\frac{3\delta \delta ^{\prime }}{1+\delta ^{2}}\right) }{\kappa
^{3}(1+f^{2})^{2}(1+\delta ^{2})^{3/2}}.
\end{equation*}%
If the numerator of the last fraction is zero, then $\left( \frac{\tau
_{\alpha }}{\kappa _{\alpha }}\right) ^{\prime }=0.$ Since the harmonic
curvature of the curve $\gamma $ is constant, it is a general helix.
\end{proof}

\begin{corollary}
Let $\beta $ be a curve with arc-length parameter $s$ and $\overline{\beta }$
be the W-direction curve of $\beta .$ The tangent indicatrix curve of $%
\overline{\beta }$ \ is a general helix if and only if the binormal
indicatrix curve of $\overline{\beta }$ is general helix.
\end{corollary}

\begin{proof}
By the equations (3.5) and (3.10), the result is clear.
\end{proof}

\begin{theorem}
The binormal indicatrix curve $\gamma $ of the W-direction curve is a slant
helix if and only if the following equation is satisfied;%
\begin{equation*}
3\delta \delta ^{\prime }\frac{X}{Y}+(1+\delta ^{2})\left( \frac{X}{Y}%
\right) ^{\prime }=0
\end{equation*}%
where $\delta =$\ $\frac{\overline{\kappa }}{\overline{\tau }}$ \ , \ $%
X=(1+\delta ^{2})\overline{\tau }(\delta ^{\prime \prime }\overline{\tau }%
-\delta ^{\prime }\overline{\tau }^{\prime })-3\delta (\delta ^{\prime })^{2}%
\overline{\tau }^{2}$ $\ $and $\ Y=\left( \overline{\tau }^{2}(1+\delta
^{2})^{3}+(\delta ^{\prime })^{2}\right) ^{3/2}.$
\end{theorem}

\begin{proof}
From the equation (1.1), the geodesic curvature of the principal image of
the principal normal indicatrix of the binormal indicatrix curve $\gamma $
is given by%
\begin{equation}
\delta _{\gamma }=\frac{\kappa _{\gamma }^{2}}{(\kappa _{\gamma }^{2}+\tau
_{\gamma }^{2})^{\frac{3}{2}}}.\left( \frac{\tau _{\gamma }}{\kappa _{\gamma
}}\right) ^{\prime }.  \tag{3.11}
\end{equation}%
We take into account that \ $\overline{\tau }^{\prime }\overline{\kappa }-%
\overline{\tau }\overline{\kappa }^{\prime }=-\overline{\tau }^{2}\delta
^{\prime }$ \ , \ \ $\overline{\kappa }^{\prime }\frac{\overline{\tau }^{2}}{%
\overline{\kappa }}-\overline{\tau }\overline{\tau }^{\prime }=\overline{%
\tau }^{2}\frac{\delta ^{\prime }}{\delta }$ \ and \ $\overline{\kappa }^{2}+%
\overline{\tau }^{2}=\overline{\tau }^{2}(1+\delta ^{2})$\ , then put the
equations (3.8) and (3.9) in (3.11), we clearly find%
\begin{equation*}
\delta _{\gamma }=\frac{\overline{\tau }^{3}(1+\delta ^{2})^{4}}{\left(
(1+\delta ^{2})^{3}\overline{\tau }^{2}+(\delta ^{\prime })^{2}\right) ^{3/2}%
}\left( \frac{\delta ^{\prime }}{\overline{\tau }(1+\delta ^{2})^{3/2}}%
\right) ^{\prime }.
\end{equation*}%
After some calculations we have%
\begin{equation}
\delta _{\gamma }=\frac{(1+\delta ^{2})^{3/2}.X}{Y}  \tag{3.12}
\end{equation}%
where $X=(1+\delta ^{2})\overline{\tau }(\delta ^{\prime \prime }\overline{%
\tau }-\delta ^{\prime }\overline{\tau }^{\prime })-3\delta (\delta ^{\prime
})^{2}\overline{\tau }^{2}$ \ \ and $\ Y=\left( \overline{\tau }%
^{2}(1+\delta ^{2})^{3}+(\delta ^{\prime })^{2}\right) ^{3/2}.$

By differentiating the equation (3.12), we obtain%
\begin{equation*}
\left( \delta _{\gamma }\right) ^{\prime }=(1+\delta ^{2})^{1/2}\left(
3\delta \delta ^{\prime }\frac{X}{Y}+(1+\delta ^{2})\left( \frac{X}{Y}%
\right) ^{\prime }\right) .
\end{equation*}%
$\left( \delta _{\gamma }\right) ^{\prime }=0$ \ if and only if $3\delta
\delta ^{\prime }\frac{X}{Y}+(1+\delta ^{2})\left( \frac{X}{Y}\right)
^{\prime }=0.$

So by taking into consideration the equation (1.1) , the proof is completed.
\end{proof}

\begin{corollary}
If any unit speed curve $\beta $ is slant helix, then the tangent indicatrix
curve of W-direction curve of $\beta $ is slant helix if and only if the
binormal indicatrix curve of W-direction curve of $\beta $ is slant helix.
\end{corollary}

\begin{proof}
By the equations (3.7) and (3.12), we have%
\begin{equation*}
\delta _{\alpha }=\delta .\delta _{\gamma }
\end{equation*}%
where $\delta =\frac{\overline{\kappa }}{\overline{\tau }}.$

If any unit speed curve $\beta $ is slant helix,then $\delta $ is constant
by Theorem 1. If $\delta $ is constant, the result is apparent.
\end{proof}

\textbf{Example : }Let a curve which is a slant helix be 
\begin{equation*}
\beta (s)=\left( -\frac{3}{2}\cos \left( \frac{s}{2}\right) -\frac{1}{6}\cos
\left( \frac{3s}{2}\right) ,-\frac{3}{2}\sin \left( \frac{s}{2}\right) -%
\frac{1}{6}\sin \left( \frac{3s}{2}\right) ,\sqrt{3}\cos \left( \frac{s}{2}%
\right) \right) .
\end{equation*}%
The tangent, binormal vectors, the curvature, the torsion and the Darboux
vector were found in \cite{Mac}.%
\begin{eqnarray*}
T(s) &=&\left( \frac{3}{4}\sin \left( \frac{s}{2}\right) +\frac{1}{4}\sin
\left( \frac{3s}{2}\right) ,-\frac{3}{4}\cos \left( \frac{s}{2}\right) -%
\frac{1}{4}\cos \left( \frac{3s}{2}\right) ,-\frac{\sqrt{3}}{2}\sin \left( 
\frac{s}{2}\right) \right)  \\
B(s) &=&\left( -\frac{1}{2}\cos \left( \frac{s}{2}\right) \left( 2\cos
^{2}\left( \frac{s}{2}\right) -3\right) ,\sin ^{3}\left( \frac{s}{2}\right) ,%
\frac{\sqrt{3}}{2}\cos \left( \frac{s}{2}\right) \right) 
\end{eqnarray*}%
and%
\begin{equation*}
\kappa (s)=\frac{\sqrt{3}}{2}\cos \left( \frac{s}{2}\right) \text{ , \ \ \ \
\ }\tau (s)=-\frac{\sqrt{3}}{2}\sin \left( \frac{s}{2}\right) .
\end{equation*}%
Also the W-direction curve of $\beta $ was given as;%
\begin{equation*}
\overline{\beta }(s)=\left( -\frac{9s}{8}-6\sin \left( \frac{s}{2}\right) -%
\frac{3}{4}\sin \left( s\right) -\frac{1}{16}\sin (2s),-\frac{1}{2}\cos (s),%
\frac{\sqrt{3}s}{2}\right) +(c_{1},c_{2},c_{3})
\end{equation*}%
where $c_{1},c_{2},c_{3}$ are constants.

Now lets find the tangent and binormal indicatrix curves of this W-direction
curve $\overline{\beta }.$ By using these expressions above and the
equations (2.3), (3.1) and (3.2), the tangent indicatrix and binormal
indicatrix are obtained respectively;%
\begin{eqnarray*}
\alpha (s_{\alpha }) &=&(-\frac{3}{4}\sin ^{2}\left( \frac{s}{2}\right) -%
\frac{1}{4}\sin \left( \frac{s}{2}\right) \sin \left( \frac{3s}{2}\right) -%
\frac{1}{2}\cos ^{2}\left( \frac{s}{2}\right) \left( 2\cos ^{2}\left( \frac{s%
}{2}\right) -3\right)  \\
&&,\text{ }\sin \left( \frac{s}{2}\right) \left[ \frac{3}{4}\cos \left( 
\frac{s}{2}\right) +\frac{1}{4}\cos \left( \frac{3s}{2}\right) +\cos \left( 
\frac{s}{2}\right) \sin ^{2}\left( \frac{s}{2}\right) \right] ,\text{ \ }%
\frac{\sqrt{3}}{2})
\end{eqnarray*}%
and%
\begin{eqnarray*}
\gamma (s_{\gamma }) &=&(\frac{\sqrt{3}}{2}\left[ \frac{3}{4}\cos ^{2}\left( 
\frac{s}{2}\right) +\frac{1}{4}\cos \left( \frac{3s}{2}\right) \cos \left( 
\frac{s}{2}\right) -\sin ^{4}\left( \frac{s}{2}\right) \right]  \\
&&,\text{ \ }\frac{\sqrt{3}}{2}\cos \left( \frac{s}{2}\right) \left[ \frac{3%
}{4}\sin \left( \frac{s}{2}\right) +\frac{1}{4}\sin \left( \frac{3s}{2}%
\right) -\frac{1}{2}\sin \left( \frac{s}{2}\right) \left( 2\cos ^{2}\left( 
\frac{s}{2}\right) -3\right) \right]  \\
&&,\text{ \ }\frac{3}{4}\cos \left( s\right) -\frac{1}{4}\left[ \sin
^{3}\left( \frac{s}{2}\right) \sin \left( \frac{3s}{2}\right) -\cos
^{3}\left( \frac{s}{2}\right) \cos \left( \frac{3s}{2}\right) \right]  \\
&&+\frac{3}{8}\cos \left( \frac{s}{2}\right) \left[ 3\cos \left( \frac{s}{2}%
\right) +\cos \left( \frac{3s}{2}\right) \right] ).
\end{eqnarray*}

\section{\textbf{Some} \textbf{Ruled Sufaces Related To W-direction Curves}}

In this section, we will identify some special ruled surfaces which are
formed by using the base curve as the W-direction curve.

Let $\beta $ be a curve with the arc-length parameter $s$ and $\overline{%
\beta }$ \ be the W-direction curve of $\beta .$ From the equations (2.4)
and (2.5), the \textit{normal }and \textit{binormal} \textit{surfaces} of $%
\overline{\beta }$ \ are given by;%
\begin{eqnarray}
\Phi _{1}(s,v) &=&\overline{\beta }(s)+v\overline{N}(s)  \TCItag{4.1} \\
\Phi _{2}(s,v) &=&\overline{\beta }(s)+v\overline{B}(s).  \TCItag{4.2}
\end{eqnarray}

\begin{theorem}
The normal surface and binormal surface of the W-direction curve are not
developable.
\end{theorem}

\begin{proof}
By using equation (2.6), the distribution parameters of the normal and the
binormal surfaces of W-direction curve given in (4.1) and (4.2) are;%
\begin{eqnarray*}
\lambda _{\Phi _{1}} &=&\frac{\det (\frac{d\overline{\beta }}{ds},\overline{N%
},\frac{d\overline{N}}{ds})}{\left\Vert \frac{d\overline{N}}{ds}\right\Vert
^{2}} \\
\lambda _{\Phi _{2}} &=&\frac{\det (\frac{d\overline{\beta }}{ds},\overline{B%
},\frac{d\overline{B}}{ds})}{\left\Vert \frac{d\overline{B}}{ds}\right\Vert
^{2}}.
\end{eqnarray*}%
Taking into account that the equations given in (2.3) and%
\begin{eqnarray*}
\frac{d\overline{N}}{ds} &=&-\left( \frac{\kappa }{\sqrt{\kappa ^{2}+\tau
^{2}}}\right) ^{\prime }T-\sqrt{\kappa ^{2}+\tau ^{2}}N+\left( \frac{\tau }{%
\sqrt{\kappa ^{2}+\tau ^{2}}}\right) ^{\prime }B \\
\frac{d\overline{B}}{ds} &=&\kappa T-\tau B
\end{eqnarray*}%
we have;%
\begin{eqnarray*}
\lambda _{\Phi _{1}} &=&\frac{\left( \kappa ^{2}+\tau ^{2}\right) ^{5/2}}{%
\left( \tau \kappa ^{\prime }-\tau ^{\prime }\kappa \right) ^{2}+\left(
\kappa ^{2}+\tau ^{2}\right) ^{3}} \\
\lambda _{\Phi _{2}} &=&\frac{1}{\sqrt{\kappa ^{2}+\tau ^{2}}}.
\end{eqnarray*}%
Since the distribution parameters cannot be zero, the normal and binormal
surfaces of W-direction curve are not developable.
\end{proof}

\begin{theorem}
If any curve $\beta $ with arc-length parameter $s$ is general helix and the
equation \ $\kappa \kappa ^{\prime }+\tau \tau ^{\prime }=0$ \ is satisfied
then the normal surface and the binormal surface of W-direction curve $%
\overline{\beta }$ \ of \ $\beta $ \ are minimal.
\end{theorem}

\begin{proof}
Lets find the mean curvatures of the normal and binormal surfaces in (4.1)
and (4.2) for minimallity.

For the normal surfaces given in (4.1), by taking into consideration the
equations (2.3), the following equations are obtained as;%
\begin{eqnarray*}
E_{1} &=&\left\langle (\Phi _{1})_{s},(\Phi _{1})_{s}\right\rangle
=A^{2}+C^{2}+D^{2} \\
F_{1} &=&\left\langle (\Phi _{1})_{s},(\Phi _{1})_{v}\right\rangle =-AX+DY=0
\\
G_{1} &=&\left\langle (\Phi _{1})_{v},(\Phi _{1})_{v}\right\rangle =1 \\
e_{1} &=&\left\langle (\Phi _{1})_{ss},U_{1}\right\rangle =\frac{1}{Z_{1}}%
\left( 
\begin{array}{c}
-(A^{\prime }+C\kappa )CY-(AY+DX)(A\kappa -C^{\prime }-D\tau ) \\ 
-(D^{\prime }-C\tau )CX%
\end{array}%
\right) \\
\ f_{1} &=&\left\langle (\Phi _{1})_{sv},U_{1}\right\rangle =\frac{1}{Z_{1}}%
\left( C(X^{\prime }Y-Y^{\prime }X)+\sqrt{\kappa ^{2}+\tau ^{2}}%
(AY+DX)\right) \\
g_{1} &=&\left\langle (\Phi _{1})_{vv},U_{1}\right\rangle =0
\end{eqnarray*}%
where $X=\frac{\kappa }{\sqrt{\kappa ^{2}+\tau ^{2}}},$ \ $Y=\frac{\tau }{%
\sqrt{\kappa ^{2}+\tau ^{2}}},$ \ $Z_{1}=\left\Vert (\Phi _{1})_{s}\times
(\Phi _{1})_{v}\right\Vert ,$ $\ U_{1}=\frac{1}{Z_{1}}((-CY)\mathbf{T}%
-(AY+DX)\mathbf{N}-(CX)\mathbf{B})$, $\ A=Y-vX^{\prime },$ \ $C=v\sqrt{%
\kappa ^{2}+\tau ^{2}}$ \ and \ $D=X+vY^{\prime }.$

For the binormal surfaces given in (4.2);%
\begin{eqnarray*}
E_{2} &=&\left\langle (\Phi _{2})_{s},(\Phi _{2})_{s}\right\rangle
=1+v^{2}(\kappa ^{2}+\tau ^{2}) \\
F_{2} &=&\left\langle (\Phi _{2})_{s},(\Phi _{2})_{v}\right\rangle =0 \\
G_{2} &=&\left\langle (\Phi _{2})_{v},(\Phi _{2})_{v}\right\rangle =1 \\
e_{2} &=&\left\langle (\Phi _{2})_{ss},U_{2}\right\rangle =\frac{1}{Z_{2}}%
\left( (\frac{\tau }{\kappa })^{\prime }\frac{1}{1+(\frac{\tau }{\kappa }%
)^{2}}+v.\frac{\kappa \kappa ^{\prime }+\tau \tau ^{\prime }}{\sqrt{\kappa
^{2}+\tau ^{2}}}+v^{2}\kappa ^{2}(\frac{\tau }{\kappa })^{\prime }\right)  \\
\ f_{2} &=&\left\langle (\Phi _{2})_{sv},U_{2}\right\rangle =\frac{\sqrt{%
\kappa ^{2}+\tau ^{2}}}{Z_{2}} \\
g_{2} &=&\left\langle (\Phi _{2})_{vv},U_{2}\right\rangle =0
\end{eqnarray*}%
where \ $Z_{2}=\sqrt{1+v^{2}(\kappa ^{2}+\tau ^{2})}$ \ and \ $U_{2}=\frac{1%
}{Z_{2}}\left( (X-v\tau )\mathbf{T}-(Y+v\kappa )\mathbf{B}\right) .$

By using the equation (2.9), we find the mean curvatures as;%
\begin{eqnarray*}
H_{1} &=&\frac{e_{1}}{2E_{1}} \\
H_{2} &=&\frac{e_{2}}{2E_{2}}.
\end{eqnarray*}%
$e_{2}$ is apparent from the above equation, lets find $e_{1}.$We can
calculate simply that%
\begin{eqnarray*}
AY+DX &=&1+v(\frac{\tau }{\kappa })^{\prime }\frac{1}{1+(\frac{\tau }{\kappa 
})^{2}} \\
A\kappa -C^{\prime }-D\tau &=&-v.\frac{\kappa \kappa ^{\prime }+\tau \tau
^{\prime }}{\sqrt{\kappa ^{2}+\tau ^{2}}} \\
-A^{\prime }CY-C^{2}\kappa Y-D^{\prime }CX+C^{2}\tau X &=&v^{2}.\frac{2(\tau
^{\prime }\kappa -\tau \kappa ^{\prime })(\kappa \kappa ^{\prime }+\tau \tau
^{\prime })-((\tau ^{\prime }\kappa )^{\prime }-(\tau \kappa ^{\prime
})^{\prime })(\kappa ^{2}+\tau ^{2})}{(\kappa ^{2}+\tau ^{2})^{3/2}}.
\end{eqnarray*}%
If we put these expressions in the equation of $e_{1}$, we have finally;%
\begin{equation*}
e_{1}=\frac{1}{Z_{1}}\left( 
\begin{array}{c}
v^{2}.\frac{2(\tau ^{\prime }\kappa -\tau \kappa ^{\prime })(\kappa \kappa
^{\prime }+\tau \tau ^{\prime })-((\tau ^{\prime }\kappa )^{\prime }-(\tau
\kappa ^{\prime })^{\prime })(\kappa ^{2}+\tau ^{2})}{(\kappa ^{2}+\tau
^{2})^{3/2}} \\ 
+v\left( 1+v(\frac{\tau }{\kappa })^{\prime }\frac{1}{1+(\frac{\tau }{\kappa 
})^{2}}\right) .\frac{\kappa \kappa ^{\prime }+\tau \tau ^{\prime }}{\sqrt{%
\kappa ^{2}+\tau ^{2}}}%
\end{array}%
\right) .
\end{equation*}

Since the curve $\beta $ is general helix, then $(\frac{\tau }{\kappa }%
)^{\prime }=0$ and so we have $(\tau ^{\prime }\kappa )^{\prime }-(\tau
\kappa ^{\prime })^{\prime }=0.$ Also by the condition $\kappa \kappa
^{\prime }+\tau \tau ^{\prime }=0$ given in theorem, we obtain $e_{1}=0$ \
and \ $e_{2}=0$.

So the proof is completed.
\end{proof}

\begin{theorem}
Let $\beta $ be any curve with arc-length parameter $s.$ If $\beta $ is
general helix, then the W-direction curve $\overline{\beta }$ which is on
the normal surface of the W-direction curve is a geodesic curve. Also the
W-direction curve $\overline{\beta }$ which is on the binormal surface is
geodesic curve.
\end{theorem}

\begin{proof}
Lets find the geodesic curvatures of the W-direction curve $\overline{\beta }
$ with respect to the normal and binormal surfaces of $\overline{\beta }$
given in (4.1) and (4.2)$.$ By using the equation (2.10), the geodesic
curvatures are%
\begin{eqnarray*}
\kappa _{g_{1}} &=&\left\langle U_{1}\times \overline{T},\overline{T}%
^{\prime }\right\rangle \\
\kappa _{g_{2}} &=&\left\langle U_{2}\times \overline{T},\overline{T}%
^{\prime }\right\rangle .
\end{eqnarray*}%
By using the same terminology with the previous proof and $\overline{T}=Y%
\mathbf{T}+X\mathbf{B}$, we have 
\begin{eqnarray*}
U_{1}\times \overline{T} &=&\frac{1}{Z_{1}}(AY+DX)(Y\mathbf{B}-X\mathbf{T})
\\
U_{2}\times \overline{T} &=&-\frac{1}{Z_{2}}\mathbf{N.}
\end{eqnarray*}%
Deciding the equation $\overline{T}^{\prime }=Y^{\prime }\mathbf{T}%
+X^{\prime }\mathbf{B}$, we get%
\begin{eqnarray*}
\kappa _{g_{1}} &=&-\frac{1}{Z_{1}}(AY+DX)(Y^{\prime }X-YX^{\prime }) \\
\kappa _{g_{2}} &=&0.
\end{eqnarray*}%
By using appropriate expressions and $Y^{\prime }X-YX^{\prime }=\frac{%
{\LARGE \tau }^{\prime }{\LARGE \kappa -\tau \kappa }^{\prime }}{{\LARGE %
\kappa }^{2}{\LARGE +\tau }^{2}}$, \ we get finally that;%
\begin{equation*}
\kappa _{g_{1}}=-\frac{1}{Z_{1}}\left( 1+v(\frac{\tau }{\kappa })^{\prime }%
\frac{1}{1+(\frac{\tau }{\kappa })^{2}}\right) (\frac{\tau }{\kappa }%
)^{\prime }\frac{1}{1+(\frac{\tau }{\kappa })^{2}}.
\end{equation*}%
If $\beta $ is general helix, then $(\frac{\tau }{\kappa })^{\prime }=0$ and 
$\kappa _{g_{1}}=0.$
\end{proof}

\begin{theorem}
Let $\beta $ be any curve with arc-length parameter $s.$ The W-direction
curve $\overline{\beta }$ which is on the normal surface of the W-direction
curve is an asymptotic line. $\beta $ is general helix if and only if the
W-direction curve $\overline{\beta }$ which is on the binormal surface is an
asymptotic line.
\end{theorem}

\begin{proof}
The normal curvatures of $\overline{\beta }$ with respect to the normal and
binormal surfaces in (4.1) and (4.2) are computed by the equation (2.11) as;%
\begin{eqnarray*}
\kappa _{n_{1}} &=&\left\langle \overline{\beta }^{\prime \prime
},U_{1}\right\rangle \\
\kappa _{n_{2}} &=&\left\langle \overline{\beta }^{\prime \prime
},U_{2}\right\rangle .
\end{eqnarray*}%
We can write that; \ $\overline{\beta }^{\prime \prime }=\overline{T}%
^{\prime }=Y^{\prime }\mathbf{T}+X^{\prime }\mathbf{B}$ and also by making
the dot product of \ $\overline{\beta }^{\prime \prime }$ and $U_{1},U_{2}$,
we get%
\begin{eqnarray*}
\kappa _{n_{1}} &=&-\frac{C}{Z_{1}}(XX^{\prime }+YY^{\prime }) \\
\kappa _{n_{2}} &=&\frac{1}{Z_{2}}(Y^{\prime }X-YX^{\prime }).
\end{eqnarray*}%
Since $XX^{\prime }+YY^{\prime }=0$ \ and \ $Y^{\prime }X-YX^{\prime }=\frac{%
{\LARGE \tau }^{\prime }{\LARGE \kappa -\tau \kappa }^{\prime }}{{\LARGE %
\kappa }^{2}{\LARGE +\tau }^{2}}$, we simply obtain that;%
\begin{eqnarray*}
\kappa _{n_{1}} &=&0 \\
\kappa _{n_{2}} &=&\frac{1}{Z_{2}}(\frac{\tau }{\kappa })^{\prime }\frac{1}{%
1+(\frac{\tau }{\kappa })^{2}}.
\end{eqnarray*}%
So $\overline{\beta }$ which is on the normal surface of the W-direction
curve is an asymptotic line. $\beta $ is general helix if and only if $(%
\frac{\tau }{\kappa })^{\prime }=0$ which means that $\kappa _{n_{2}}=0$ and
it is an asymtotic line.
\end{proof}

\begin{theorem}
Let $\beta $ be any curve with arc-length parameter $s.$ If $\beta $ is
general helix, then the W-direction curve $\overline{\beta }$ which is on
the normal surface and the binormal surface of the W-direction curve is a
principal line.
\end{theorem}

\begin{proof}
Lets find the geodesic torsions of the W-direction curve $\overline{\beta }$
with respect to the normal and binormal surfaces given in (4.1) and (4.2).
By using the equation (2.12), the geodesic torsions are;%
\begin{eqnarray*}
\tau _{g_{1}} &=&\left\langle U_{1}\times U_{1}^{\prime },\overline{T}%
^{\prime }\right\rangle \\
\tau _{g_{2}} &=&\left\langle U_{2}\times U_{2}^{\prime },\overline{T}%
^{\prime }\right\rangle .
\end{eqnarray*}%
After some calculations and using the abbreviations as \ $r=-\frac{{\LARGE CY%
}}{{\LARGE Z}_{1}}$, \ $q=-\frac{{\LARGE AY+DX}}{{\LARGE Z}_{1}}$, \ $t=-%
\frac{{\LARGE CX}}{{\LARGE Z}_{1}}$ \ \ and \ \ $K=\left( \frac{{\LARGE %
X-v\tau }}{{\LARGE Z}_{2}}\right) ^{\prime },$ \ $L=\frac{{\LARGE %
X+Y+v(\kappa -\tau )}}{{\LARGE Z}_{2}},$ \ $M=\left( \frac{{\LARGE Y+v\kappa 
}}{{\LARGE Z}_{2}}\right) ^{\prime },$ we find;%
\begin{eqnarray*}
U_{1}\times U_{1}^{\prime } &=&\left( \frac{CX}{Z_{1}}(r\kappa +q^{\prime
}-t\tau )-\frac{AY+DX}{Z_{1}}(q\tau +t^{\prime })\right) \mathbf{T} \\
&&+\left( \frac{CY}{Z_{1}}(q\tau +t^{\prime })-\frac{CX}{Z_{1}}(r^{\prime
}-q\kappa )\right) \mathbf{N} \\
&&+\left( \frac{AY+DX}{Z_{1}}(r^{\prime }-q\kappa )-\frac{CY}{Z_{1}}(r\kappa
+q^{\prime }-t\tau )\right) \mathbf{B}
\end{eqnarray*}%
and%
\begin{eqnarray*}
U_{2}\times U_{2}^{\prime } &=&\left( \frac{Y+v\kappa }{Z_{2}}\right) L.%
\mathbf{T+}\left( \left( \frac{X-v\tau }{Z_{2}}\right) M-\left( \frac{%
Y+v\kappa }{Z_{2}}\right) K\right) \mathbf{N} \\
&&+\left( \frac{X-v\tau }{Z_{2}}\right) L.\mathbf{B}
\end{eqnarray*}%
By taking into account that $\overline{T}^{\prime }=Y^{\prime }\mathbf{T}%
+X^{\prime }\mathbf{B}$, we obtain lastly;%
\begin{eqnarray*}
\tau _{g_{1}} &=&\frac{1}{Z_{1}}\left( \frac{C}{Z_{1}}\right) ^{\prime }(%
\frac{\tau }{\kappa })^{\prime }\frac{1}{1+(\frac{\tau }{\kappa })^{2}}%
\left( 1+v(\frac{\tau }{\kappa })^{\prime }\frac{1}{1+(\frac{\tau }{\kappa }%
)^{2}}\right) \\
&&+\frac{C(r\kappa +q^{\prime }-t\tau )}{Z_{1}}(\frac{\tau }{\kappa }%
)^{\prime }\frac{1}{1+(\frac{\tau }{\kappa })^{2}} \\
\tau _{g_{2}} &=&\frac{v\kappa ^{2}(\kappa +\tau +v(\kappa -\tau )\sqrt{%
\kappa ^{2}+\tau ^{2}})}{(\kappa ^{2}+\tau ^{2})(1+v^{2}(\kappa ^{2}+\tau
^{2}))}.(\frac{\tau }{\kappa })^{\prime }
\end{eqnarray*}%
If $\beta $ is general helix, then $(\frac{\tau }{\kappa })^{\prime }=0$ and
so $\tau _{g_{1}}=0,$ \ $\tau _{g_{2}}=0.$
\end{proof}

\end{document}